\documentclass[a4paper,11pt]{article}
\usepackage[T2A]{fontenc}
\usepackage[cp1251]{inputenc}
\usepackage[ukrainian, english]{babel}
\usepackage{amssymb,amsfonts,amsthm,amsmath,mathtext,cite,enumerate,float}
\usepackage[onehalfspacing]{setspace}
\usepackage{geometry}
\usepackage{graphicx}
\usepackage{multirow}
\usepackage{indentfirst}
\usepackage{color}
\usepackage{comment}

\newcommand{\aut}[1][n]{B_{#1}}
\newcommand{\syl}[1][n]{G_{#1}}

\newcommand{\m}{^{-1}}

\newcommand{\rest}[1][n]{_{(#1)}}
\newcommand{\ItWr}[2]{\stackrel{#2}{ \underset{\text{\it #1=1}}{\wr }}}

\newif\ifbrief
\brieffalse

\ifbrief
\else
\newtheorem{theorem}{Theorem}
\newtheorem{proposition}[theorem]{Proposition}
\newtheorem{corollary}[theorem]{Corollary}
\newtheorem{lemma}[theorem]{Lemma}
\newtheorem{example}{Example}
\newtheorem{remark}[theorem]{Remark}

\theoremstyle{definition}
\newtheorem{definition}{Definition}

\newtheorem{statment}{Statement}
\fi

\begin{document}
\title{The commutator subgroup of Sylow 2-subgroups of alternating group, commutator width of wreath product }
\author{Ruslan Skuratovskii} 
\date{\empty}
\maketitle



  \begin{abstract}
  We construct the minimal generating set of  the commutator subgroup of Sylow 2-subgroup of alternating group. Inclusion problem \cite{Lin} for $Sy{{l}_{2}}{{A}_{{{2}^{k}}}}$ and its subgroups as $\left( Sy{{l}_{2}}{{A}_{{{2}^{k}}}} \right)'$ and $\left( Sy{{l}_{2}}{{A}_{{{2}^{k}}}} \right)''$ is investigated by us. Relation between solving of inclusion problem of and conjugacy search problem \cite{Ushak} in this group is justified by us.
The minimal generating set for the commutator subgroup of Sylow 2-subgroups of alternating group ${A_{{2^{k}}}}$ was constructed in form of wreath recursion.

The size of such minimal generating set is found. The structure of commutator subgroup of Sylow 2-subgroups of the alternating group ${A_{{2^{k}}}}$ is investigated.


It is shown that $(Syl_2 A_{2^k})^2 = Syl'_2 A_{2^k}, \, k>2$.


      The commutator width of direct limit of wreath product of cyclic groups is found.
  This paper presents upper bounds of the commutator width $(cw(G))$ \cite {Mur} of a wreath product of groups.

A new approach to presentation of Sylow 2-subgroups of the alternating group ${A_{{2^{k}}}}$ is applied.
As a result the short proof that the commutator width of Sylow 2-subgroups of alternating group ${A_{{2^{k}}}}$, permutation group ${S_{{2^{k}}}}$ and Sylow $p$-subgroups of $Syl_2 A_{p^k}$ ($Syl_2 S_{p^k}$) are equal to 1 is obtained.

 An upper bound of the commutator width of permutational wreath product $B \wr C_n$ for an arbitrary group $B$  is found.




  \end{abstract}


  \begin{section}{Introduction}
  The first example of a group $G$ with $cw(G) > 1$ was given by
Fite \cite{Fite}.  The smallest finite examples of such groups are groups of order 96, there's two of them, nonisomorphic to each other, were given by Guralnick \cite{Gural_2010}. 

  We deduce an estimation for commutator width of wreath product of groups $C_n \wr B$ taking in consideration a $cw(B)$ of passive group $B$.

The form of commutator presentation \cite{Meld} is proposed by us as wreath recursion \cite{Lav} and commutator width of it was studied. We impose more weak condition on the presentation of wreath product commutator then it was imposed by J. Meldrum.

In this paper we continue a researches which was stared in \cite{SkIrred, SkAr}. We find a minimal generating set and the structure for commutator subgroup of $Sy{l_{2}}{{A}_{{{2}^{k}}}}$.

A research of commutator-group
 serve to decision  of inclusion problem \cite{Lin} for elements of $Syl_2 {A_{{2^{k}}}}$ in its derived subgroup $(Syl_2 {A_{{2^{k}}}})'$.
It was known that, the commutator width of iterated wreath products of nonabelian finite simple groups is bounded by an absolute constant \cite {nikolov, Fite}. But it was not proven that commutator subgroup of $\underset{i=1}{\overset{k}{\mathop{\wr }}}\,{{\mathcal{C}}_{{{p}_{i}}}}$ consists of commutators. We generalize the passive group of this wreath product to any group $B$ instead of only wreath product of cyclic groups and obtain an exact commutator width.

 Also we are going to prove that the commutator width of Sylows $p$-subgroups of symmetric and alternating groups $p \geq 2$ is 1.

  \end{section}

  \begin{section}{Preliminaries }





Let $G$ be a group acting (from the right) by permutations on
a set $X$ and let $H$ be an arbitrary group.
Then the (permutational) wreath product
$H \wr G$ is the semidirect product $H^X \leftthreetimes G $, 
 where $G$ acts on the direct power $H^X$ by
the respective permutations of the direct factors.
The group $C_p$ or $(C_p, X)$ is equipped with a natural action by the left shift on $X =\{1,…,p\}$, $p\in \mathbb{N}$.
As well known that a wreath product of permutation groups is associative construction.

The multiplication rule of automorphisms $g$, $h$ which presented in form of the wreath recursion \cite{Ne}
$g=(g\rest[1],g\rest[2],\ldots,g\rest[d])\sigma_g, \
h=(h\rest[1],h\rest[2],\ldots,h\rest[d])\sigma_h,$ is given by the formula:
$$g\cdot h=(g\rest[1]h\rest[\sigma_g(1)],g\rest[2]h\rest[\sigma_g(2)],\ldots,g\rest[d]h\rest[\sigma_g(d)])\sigma_g \sigma_h.$$

We define $\sigma$ as $(1,2,\ldots, p)$ where $p$ is defined by context.

 The set $X^*$ is naturally a vertex set of a regular rooted tree, i.e. a connected graph without cycles
and a designated vertex $v_0$ called the root, in which two words are connected by an edge if and only if they are of form $v$ and $vx$, where $v\in X^*$, $x\in X$.
The set $X^n \subset X^*$ is called the $n$-th level of the tree $X^*$
and $X^0 = \{v_0\}$. We denote by $v_{ji}$ the vertex of $X^j$, which has the number $i$.
Note that the unique vertex $v_{k,i}$ corresponds to the unique word $v$ in alphabet $X$.
For every automorphism $g\in Aut{{X}^{*}}$ and every word $v \in X^{*}$  define the section (state) $g_{(v)} \in AutX^{*}$ of $g$ at $v$ by the rule: $g_{(v)}(x) = y$ for $x, y \in X^*$  if and only if $g(vx) = g(v)y$.
The subtree of $X^{*}$ induced by the set of vertices $\cup_{i=0}^k X^i$ is denoted by $X^{[k]}$.
 The restriction of the action of an automorphism $g\in AutX^*$ to the subtree $X^{[l]}$ is denoted by $g_{(v)}|_{X^{[l]}}$.
 A restriction $g_{(v_{ij})}|_{X^{[1]}} $ is called the vertex permutation (v.p.) of $g$ in a vertex $v_{ij}$ and denoted by $g_{ij}$.
We call the endomorphism $\alpha|_{v} $ restriction of $g$ in a vertex $v$ \cite{Ne}. For example, if $|X| = 2$ then we just have to distinguish active vertices, i.e., the
vertices for which $\alpha|_{v} $ is non-trivial.

Let us label every vertex of ${{X}^{l}},\,\,\,0\le l<k$ by sign 0 or 1 in relation to state of v.p. in it. 
Obtained by such way a vertex-labeled regular tree is an element of $Aut{{X}^{[k]}}$.
All undeclared terms are from \cite{Sam, Gr}.

Let us make some notations.
For brevity, in form of wreath recursion we write a commutator as $[a,b]=ab{{a}^{-1}}{{b}^{-1}}$ that is inverse to ${{a}^{-1}}{{b}^{-1}}ab$. That does not reduce the generality of our reasoning.
Since for convenience the commutator of two group elements $a$ and $b$ is denoted by
$
[a,b] = aba\m b\m,
$
 conjugation by an element $b$ as
$
a^b = bab\m.
$

We define $G_k$ and $B_k$ recursively i.e.
\begin{align*}
B_1 &= C_2, \, B_k = B_{k-1} \wr C_2  \mbox{ for $k>1$},&\\
G_1 &= \langle e \rangle, \, G_k =  \{(g_1, g_2)\pi \in B_{k} \mid g_1g_2 \in G_{k-1} \} \mbox{ for $k>1$.}&
\end{align*}

Note that $B_k = \ItWr{i}{k} C_2 $.


We denoted by \emph{clG(g)} the commutator length of an element $g$ of the derived
subgroup of a group $G$ is the minimal $n$ such that there
exist elements $x_1, \ldots , x_n, y_1, \ldots , y_n$ in G such that $g = [x_1, y_1] \ldots [x_n, y_n]$.
The commutator length of the identity element is 0. The commutator width
of a group $G$, denoted $cw(G)$, is the maximum of the commutator lengths
of the elements of its derived subgroup $[G,G]$.
The minimal number of generators of the group $G$ is denoted by $d(G)$.

  \end{section}

  \begin{section}{Commutator width of Sylow 2-subgroups of $A_{2^k}$ and $S_{2^k}$ }

The following Lemma imposes the Corollary 4.9 of \cite{Meld} and it will be deduced from the corollary 4.9 with using in presentation elements in the form of wreath recursion.
\begin{lemma}\label{form of comm} An element of form
$(r_1, \ldots, r_{p-1}, r_p) \in W'= (B \wr C_p)'$ iff product of all $r_i$ (in any order) belongs to $B'$, where $p\in \mathrm{N}$, $p\geq2$.
\end{lemma}
\begin{proof}
More details of our argument may be given as follows.
\begin{eqnarray*}
w=(r_1, r_2, \ldots, r_{p-1},  r_p),
\end{eqnarray*}
where $r_i\in B$.
If we multiply elements from a tuple $(r_1, \ldots, r_{p-1}, r_p)$, where $r_i={{h}_{i}}{{g}_{a(i)}}h_{ab(i)}^{-1}g_{ab{{a}^{-1}}(i)}^{-1}$, $h, \, g \in B$ and $a,b \in C_p$, then we get a product
\begin{equation} \label{Meld} 
 x=\stackrel{p}{ \underset{\text{\it i=1}} \prod} r_i = \prod\limits_{i=1}^{p}{{{h}_{i}}{{g}_{a(i)}}h_{ab(i)}^{-1}g_{ab{{a}^{-1}}(i)}^{-1} \in B'},
\end{equation}
where $x $ is a product of corespondent commutators.
Therefore, we can write $r_p = r_{p-1}\m \ldots r_1\m x$. We can rewrite element $x\in B'$ as the product $x = \prod \limits ^{m}_{j=1} [f_j,g_j]$,  $m \le cw(B)$.

Note that we impose more weak condition on the product of all $r_i$ to belongs to $B'$ then in Definition 4.5. of form $P(L)$ in \cite{Meld}, where the product of all $r_i$ belongs to a subgroup $L$ of $B$ such that $ L>B'$.

 In more detail deducing of our representation constructing can be reported in following way.
 If we multiply elements having form of a tuple $(r_1, \ldots, r_{p-1}, r_p)$, where  $r_i={{h}_{i}}{{g}_{a(i)}}h_{ab(i)}^{-1}g_{ab{{a}^{-1}}(i)}^{-1}$, $h, \, g \in B$ and $a,b \in C_p$, then in case $cw(B)=0$ we obtain a product
\begin{equation}\label{Meld2}
 \stackrel{p}{ \underset{\text{\it i=1}} \prod} r_i = \prod\limits_{i=1}^{p}{{{h}_{i}}{{g}_{a(i)}}h_{ab(i)}^{-1}g_{ab{{a}^{-1}}(i)}^{-1} \in B'}.
\end{equation}

Note that if we rearrange elements in (1) as $h_{1} h_{1}^{-1} g_{1}g_2^{-1}h_{2} h_{2}^{-1} g_{1}g_2^{-1} ...  h_{p} h_{p}^{-1} g_{p}g_p^{-1}$ then by the reason of such permutations we obtain a product of corespondent commutators. 
Therefore, following equality holds true

\begin{equation}\label{HH}
\prod\limits_{i=1}^{p}{{{h}_{i}}{{g}_{a(i)}}h_{ab(i)}^{-1}g_{ab{{a}^{-1}}(i)}^{-1} } =\prod\limits_{i=1}^{p}h_{i} g_{i} h_{i}^{-1} g_i^{-1}x_0 =\prod\limits_{i=1}^{p}h_{i} h_{i}^{-1} g_{i}g_i^{-1}x \in B',
\end{equation}
where $x_0,  x$ are a products of corespondent commutators.
Therefore,
\begin{eqnarray} \label{form}
(r_1, \ldots, r_{p-1}, r_p) \in W' \mbox{ iff } r_{p-1} \cdot \ldots \cdot r_{1} \cdot r_p = x\in B'.
\end{eqnarray}
 Thus, one element from states of wreath recursion $(r_1, \ldots, r_{p-1}, r_p) $ depends on rest of $r_i$. This dependence contribute that the product $\prod\limits_{j=1}^{p}r_{j}$ for an arbitrary sequence $\{ r_{j} \}_{j=1} ^{p}$
 belongs to  $B'$. Thus, $r_p$ can be expressed as:
\begin{eqnarray*}
r_p = r_{1}\m \cdot \ldots \cdot r_{p-1}\m x.
\end{eqnarray*}

Denote a $j$-th tuple, which consists of a wreath recursion elements, by $(r_{{j}_1},r_{{j}_2},..., r_{{j}_p} )$.
Closedness by multiplication of the set of forms $(r_1, \ldots, r_{p-1}, r_p) \in W= (B \wr C_p)'$
  follows from


\begin{eqnarray} \label{prod}
  \prod\limits_{j=1}^{k}  ( r_{j1} \ldots r_{j{p-1}} r_{jp})= \prod\limits_{j=1}^{k} \prod\limits_{i=1}^{p}  r_{j_i} =  R_1 R_2 ...  R_{k} \in B ',
\end{eqnarray}

  where $r_{ji}$ is $i$-th element from the tuple number $j$,  $R_j = \prod\limits_{i=1}^{p}  r_{ji}, \,\, \, 1 \leq j \leq  k$. As it was shown above $R_j = \prod\limits_{i=1}^{p-1}  r_{ji} \in B'$. Therefore, the product (\ref{prod}) of $R_j$, $j \in \{1,...,k \}$ which is similar to the product mentioned in \cite{Meld}, has the property $R_1 R_2 ...  R_{k} \in B '$ too, because of $B '$ is subgroup.
   Thus, we get a product of form (\ref{Meld}) and the similar reasoning as above are applicable.

Let us prove the sufficiency condition. 
If the set $K$ of elements satisfying the condition of this theorem, that all products of all $r_i$, where every $i$ occurs in this forms once, belong to $B'$, then using the elements of form

 $(r_{1},e,..., e, r_{1}^{-1} )$, ... , $(e,e,...,e, r_{i}, e, r_{i}^{-1} )$, ... ,$(e,e,..., e, r_{p-1}, r_{p-1}^{-1})$, $(e,e,..., e, r_1 r_2 \cdot ...\cdot r_{p-1} )$

  we can express any element of form $(r_1, \ldots, r_{p-1}, r_p) \in W= (B\wr C_p)'$. We need to prove that in such way we can express all element from $W$ and only elements of $W$. The fact that all elements can be generated by elements of $K$ follows from randomness of choice every $r_i$, $i<p$ and the fact that equality (1) holds so construction of $r_p$ is determined.
\end{proof}

\begin{lemma} \label{form of comm_2} For any group $B$ and integer $p\geq 2$  if $w\in (B \wr C_p)'$ then $w$ can be represented as the following wreath recursion
\begin{align*}
w=(r_1, r_2, \ldots, r_{p-1},  r_1\m \ldots r_{p-1}\m \prod \limits ^{k}_{j=1} [f_j,g_j]),
\end{align*}
where $r_1, \ldots, r_{p-1}, f_j, g_j \in B$ and $k\leq cw(B)$.
\end{lemma}
\begin{proof}
According to Lemma~\ref{form of comm} we have the following wreath recursion
\begin{align*}
w=(r_1, r_2, \ldots, r_{p-1},  r_p),
\end{align*}
where $r_i\in B$ and $r_{p-1} r_{p-2} \ldots r_2 r_1  r_p = x \in B'$. Therefore we can write $r_p = r_1\m \ldots r_{p-1}\m x$. We also can rewrite element $x\in B'$ as product of commutators $x = \prod \limits ^{k}_{j=1} [f_j,g_j]$ where $k\leq cw(B)$.
\end{proof}

\begin{lemma} \label{c_p_wr_b_elem_repr}
For any group $B$ and integer $p\geq 2$ if $w\in (B \wr C_p)'$ is defined by the following wreath recursion
\begin{align*}
w=(r_1, r_2, \ldots, r_{p-1},  r_1\m \ldots r_{p-1}\m [f,g]),
\end{align*}
where $r_1, \ldots, r_{p-1}, f, g \in B$ then we can represent $w$ as the following  commutator
\begin{align*}
w = [(a_{1,1},\ldots, a_{1,p})\sigma, (a_{2,1},\ldots, a_{2,p})],
\end{align*}
where
\begin{align*}
a_{1,i} &=  e, \mbox{ for $1\leq i \leq p-1$ },\\
a_{2,1} &= (f\m)^{r_1\m \ldots r_{p-1}\m},\\
a_{2,i} &= r_{i-1} a_{2,i-1},\mbox{ for $2\leq i \leq p$},\\
a_{1,p} &= g^{a_{2,p}\m}.
\end{align*}
\end{lemma}

\begin{proof}
Let us to consider the following commutator
\begin{align*}
\kappa &= (a_{1,1},\ldots, a_{1,p})\sigma \cdot (a_{2,1},\ldots, a_{2,p}) \cdot (a_{1,p}\m,a_{1,1}\m,\ldots, a_{1,p-1}\m)\sigma\m \cdot (a_{2,1}\m,\ldots, a_{2,p}\m)\\
&= (a_{3,1}, \ldots, a_{3,p}),
\end{align*}
where
\begin{align*}
a_{3,i} = a_{1,i}a_{2,1 + (i \bmod p)}a_{1,i}\m a_{2,i}\m.
\end{align*}
At first we compute the following
\begin{align*}
a_{3,i} = a_{1,i}a_{2,i+1}a_{1,i}\m a_{2,i}\m = a_{2,i+1} a_{2,i}\m = r_{i} a_{2,i} a_{2,i}\m=  r_i, \mbox{ for $1\leq i \leq p-1$}.
\end{align*}
Then we make some transformation of $a_{3,p}$:
\begin{align*}
a_{3,p}&=a_{1,p}a_{2,1}a_{1,p}\m a_{2,p}\m\\
&=(a_{2,1} a_{2,1}\m) a_{1,p}a_{2,1}a_{1,p}\m a_{2,p}\m\\
&=a_{2,1} [a_{2,1}\m, a_{1,p}] a_{2,p}\m\\
&=a_{2,1}a_{2,p}\m a_{2,p} [a_{2,1}\m, a_{1,p}] a_{2,p}\m\\
&= (a_{2,p} a_{2,1}\m)\m  [(a_{2,1}\m)^{a_{2,p}}, a_{1,p}^{a_{2,p}}]\\
&= (a_{2,p} a_{2,1}\m)\m  [(a_{2,1}\m)^{a_{2,p} a_{2,1}\m}, a_{1,p}^{a_{2,p}}].
\end{align*}
Now we can see that the form of the commutator $\kappa$ is similar to the form of $w$.

Let us make the following notation
\begin{align*}
r' = r_{p-1} \ldots r_1.
\end{align*}
We note that from the definition of $a_{2, i}$ for $2\leq i \leq p$ it follows that
\begin{align*}
r_i = a_{2,i+1} a_{2,i}\m, \mbox{ for $1\leq i \leq p-1$}.
\end{align*}
Therefore
\begin{align*}
r' &=  (a_{2,p} a_{2,p-1}\m) (a_{2,p-1} a_{2,p-2}\m)\ldots (a_{2,3} a_{2,2}\m) (a_{2,2} a_{2,1}\m)\\
&= a_{2,p} a_{2,1}\m.
\end{align*}
And then
\[
(a_{2,p} a_{2,1}\m)\m = (r')\m = r_1\m \ldots r_{p-1}\m.
\]
And now we compute the following
\begin{align*}
(a_{2,1}\m)^{a_{2,p} a_{2,1}\m} &= (((f\m)^{r_1\m \ldots r_{p-1}\m})\m)^{r'} = (f^{(r')\m})^{r'}  = f,\\
a_{1,p}^{a_{2,p}} &= (g^{a_{2,p}\m})^{a_{2,p}} = g.
\end{align*}
Finally we conclude that
\begin{align*}
a_{3,p} = r_1\m \ldots r_{p-1}\m [f,g].
\end{align*}
Thus, the commutator $\kappa$ is presented exactly in the similar form as $w$ has.
\end{proof}
For future using we formulate previous Lemma for the case $p=2$.
\begin{corollary} \label{c_2_wr_b_elem_repr}
For any group $B$ if $w\in (B \wr C_2)'$ is defined by the following wreath recursion
\begin{align*}
w=(r_1,  r_1\m [f,g]),
\end{align*}
where $r_1, f, g \in B$ then we can represent $w$ as commutator
\begin{align*}
w = [(e, a_{1,2})\sigma, (a_{2,1}, a_{2,2})],
\end{align*}
where
\begin{align*}
a_{2,1} &= (f\m)^{r_1\m},\\
a_{2,2} &= r_{1} a_{2,1},\\
a_{1,2} &= g^{a_{2,2}\m}.
\end{align*}
\end{corollary}

\begin{lemma} \label{comm B_k}
For any group $B$ and integer $p\geq 2$ inequality
\begin{align*}
cw(B \wr C_p) \leq \max(1,cw(B))
\end{align*}
holds.
\end{lemma}
\begin{proof}
We can represent any $w\in (B \wr C_p)'$ by Lemma~\ref{form of comm} with the following wreath recursion
\begin{align*}
w&=(r_1, r_2, \ldots, r_{p-1},  r_1\m \ldots, r_{p-1}\m  \prod \limits_{j=1}^{k} [f_{j},g_{j}])\\
 &= (r_1, r_2, \ldots, r_{p-1},  r_1\m \ldots, r_{p-1}\m  [f_{1},g_{1}]) \cdot \prod \limits_{j=2} ^{k} [(e, \ldots, e, f_j), (e, \ldots, e, g_j)],
\end{align*}
where $r_1, \ldots, r_{p-1}, f_j, g_j \in B$ and $k\leq cw(B)$. Now by the Lemma~\ref{c_p_wr_b_elem_repr}  we  can see that $w$ can be represented as a product of $\max(1, cw(B))$ commutators.
\end{proof}

\begin{corollary} \label{comm Cycl}
If  $W = C_{p_k} \wr \ldots \wr C_{p_1}$ then
$cw(W) =1$ for $k\geq 2$.
\end{corollary}
\begin{proof}
If $B= C_{p_k} \wr C_{p_{k-1}}$  then taking into consideration that $cw(B)>0$ (because $C_{p_k} \wr C_{p_{k-1}}$ is not commutative group). Since Lemma \ref{comm B_k} implies that $cw(C_{p_k} \wr C_{p_{k-1}})=1$ then according to the inequality $cw(C_{p_k} \wr C_{p_{k-1}} \wr C_{p_{k-2}}) \leq \max(1,cw(B))$ from Lemma \ref{comm B_k} we obtain $cw(C_{p_k} \wr C_{p_{k-1}} \wr C_{p_{k-2}})=1$. Analogously if $W = C_{p_k} \wr \ldots \wr C_{p_1}$ and supposition of induction for $C_{p_{k}} \wr \ldots \wr C_{p_2}$ holds, then using an associativity of a permutational wreath product we obtain from the inequality of Lemma \ref{comm B_k} and the equality $cw( C_{p_k} \wr \ldots \wr C_{p_2})=1$ that $cw(W)=1$.
\end{proof}



We define our partial ordered set $M$ as the set of all finite wreath products of cyclic groups. We make of use directed set $\mathbb{N}$.
\begin{equation} \label{cwH }
{{H}_{k}}=\underset{i=1}{\overset{k}{\mathop{\wr }}}\,{{\mathcal{C}}_{{{p}_{i}}}}
\end{equation}

Moreover, it has already been proved in Corollary  \ref{comm Cycl} that each group  of the form $\underset{i=1}{\overset{k}{\mathop{\wr }}}\,{{\mathcal{C}}_{{{p}_{i}}}}$  has a  commutator  width equal to 1, i.e $cw(\underset{i=1}{\overset{k}{\mathop{\wr }}}\,{{\mathcal{C}}_{{{p}_{i}}}})=1$.   A partial order  relation will be a subgroup relationship.  Define the injective homomorphism $f_{k,k+1}$ from the $\underset{i=1}{\overset{k}{\mathop{\wr}}}\,{{\mathcal{C}}_{{{p}_{i}}}}$ into $\underset{i=1}{\overset{k+1}{\mathop{\wr }}}\,{{\mathcal{C}}_{{{p}_{i}}}}$ by mapping a generator of active group ${\mathcal{C}_{{{p}_{i}}}}$ of ${{H}_{k}}$ in a generator of active group ${\mathcal{C}_{{{p}_{i}}}}$ of ${{H}_{k+1}}$.
In more details the injective homomorphism $f_{k,k+1}$ is defined as $g \mapsto g(e,..., e)$, where a generator $g\in \underset{i=1}{\overset{k} {\mathop{\wr }}}\, {\mathcal{C}}_{{p}_{i}}$, $g(e,..., e)\in  \underset{i=1}{\overset{k+1}{\mathop{\wr }}}\, {\mathcal{C}}_{{p}_{i}}$.

Therefore this is an injective homomorphism of ${{H}_{k}}$ onto subgroup $\underset{i=1}{\overset{k}{\mathop{\wr }}}\,{{\mathcal{C}}_{{{p}_{i}}}}$ of $H_{k+1}$.

\begin{corollary} \label{comm Cycl}
The direct limit  $\underrightarrow{\lim }\underset{i=1}{\overset{k}{\mathop{\wr }}}\,{{\mathcal{C}}_{{{p}_{i}}}}$ of  direct system $\left\langle {{f}_{k,j}},\,\underset{i=1}{\overset{k}{\mathop{\wr }}}\,{{\mathcal{C}}_{{{p}_{i}}}} \right\rangle $ has commutator width 1.
\end{corollary}

\begin{proof}
We make the transition to the direct limit in the direct system $\left\langle {{f}_{k,j}},\,\underset{i=1}{\overset{k}{\mathop{\wr }}}\,{{\mathcal{C}}_{{{p}_{i}}}} \right\rangle $  of injective mappings from chain $e\to \,\,...\,\,\to \underset{i=1}{\overset{k}{\mathop{\wr }}}\,{{\mathcal{C}}_{{{p}_{i}}}}\to \underset{i=1}{\overset{k+1}{\mathop{\wr }}}\,{{\mathcal{C}}_{{{p}_{i}}}}\to \underset{i=1}{\overset{k+2}{\mathop{\wr }}}\,{{\mathcal{C}}_{{{p}_{i}}}}\to ...$.


Since all mappings in chains are injective homomorphisms, it has a trivial kernel. Therefore the transition to a direct limit boundary preserves the property $cw(H)=1$, because each group ${{H}_{k}}$ from the chain endowed by $cw(H_k)=1$.


The direct limit of the direct system  is denoted by $\underrightarrow{\lim }\underset{i=1}{\overset{k}{\mathop{\wr }}}\,{{\mathcal{C}}_{{{p}_{i}}}}$ and is defined as disjoint union of the  ${{H}_{k}}$'s modulo a certain equivalence relation:

$$\underrightarrow{\lim }\underset{i=1}{\overset{k}{\mathop{\wr }}}\,{{\mathcal{C}}_{{{p}_{i}}}}={}^{\coprod\limits_{k}{\underset{i=1}{\overset{k}{\mathop{\wr }}}\,{{\mathcal{C}}_{{{p}_{i}}}}}}/{}_{\sim }.$$

Since every element $g$ of $\underrightarrow{\lim }\underset{i=1}{\overset{k}{\mathop{\wr }}}\,{{\mathcal{C}}_{{{p}_{i}}}}$ coincides with a correspondent element from some ${{H}_{k}}$ of direct system, then by the injectivity of the mappings for $g$
the property  $cw(\underset{i=1}{\overset{k}{\mathop{\wr }}}\,{{\mathcal{C}}_{{{p}_{i}}}})=1$ also holds. Thus, it holds for the whole $\underrightarrow{\lim }\underset{i=1}{\overset{k}{\mathop{\wr }}}\,{{\mathcal{C}}_{{{p}_{i}}}}$.
\end{proof}

\begin{corollary} \label{cw_syl_p_s_p_k_eq_1_and_syl_p_a_p_k_eq_1} For prime $p$ and $k\geq 2$ commutator width  $cw(Syl_p(S_{p^k})) = 1$ and for prime $p>2$ and $k\geq 2$ commutator width $cw(Syl_p (A_{p^k})) = 1$.
\end{corollary}
\begin{proof}
Since $Syl_p(S_{p^k}) \simeq  \stackrel{k}{ \underset{\text{\it i=1}}{\wr }}C_p$ see \cite{Kal, Paw}, then $cw(Syl_p(S_{p^k}))=1$. As well known in case $p>2$ we have $Syl_p S_{p^k} \simeq Syl_p A_{p^k} $ see \cite{SkIrred, Dm}, then $cw(Syl_p(A_{p^k}))=1$.
\end{proof}

\begin{proposition}\label{comm_F_k_is_subgroup_of_L_k}
The following inclusion $B_k'<G_k$ holds.
\end{proposition}
 \begin{proof}
 Induction on $k$. For $k=1$ we have $B_k' = G_k = \{ e \}$. Let us fix some  $g=(g_1, g_2)\in B_k'$. Then $g_1 g_2 \in B_{k-1}'$ by Lemma~\ref{form of comm}. As $B_{k-1}' < G_{k-1}$  by induction hypothesis therefore $g_1 g_2 \in G_{k-1}$ and by definition of $G_k$ it follows that $g \in G_k$.
%
 \end{proof}
 \begin{corollary}\label{G_k_is_normal_in_B_k}
 The set $G_k$ is a subgroup in the group $B_k$.
 \end{corollary}
 \begin{proof}
 According to recursively definition of $G_k$ and $B_k$, where
$G_k = \{(g_1, g_2)\pi \in B_{k} \mid g_1g_2 \in G_{k-1} \} \,\, k>1$,  $G_k$ is subset of $B_k$ with condition $g_1g_2 \in G_{k-1}$. It is easy to check the closedness by multiplication elements of $G_k$ with condition $g_1g_2, h_1h_2 \in G_{k-1}$ because $G_{k-1}$ is subgroup so $g_1g_2h_1h_2 \in G_{k-1}$ too. A condition of existing inverse be verified trivial.
 \end{proof}

\begin{lemma}\label{order_of_G_k}
For any $k\geq 1$ we have $|G_k| = |B_k|/2$.
\end{lemma}
\begin{proof} Induction on $k$. For $k=1$ we have $|G_1| = 1 = |B_1/2|$.
Every element $g\in G_k$ can be uniquely write as the following wreath recursion
\[
g = (g_1, g_2)\pi = (g_1, g_1\m x)\pi
\]
where $g_1 \in B_{k-1}$, $x\in G_{k-1}$ and $\pi \in C_2$. Elements $g_1, x$ and $\pi$ are independent therefore $|G_k| = 2  |B_{k-1}| \cdot |G_{k-1}| = 2  |B_{k-1}| \cdot |B_{k-1}|/2= |B_k|/2$.
\end{proof}

\begin{corollary}\label{G_k_is_normal_in_B_k}
The group $G_k$ is a normal subgroup in the group $B_k$ i.e. $ G_k \lhd B_k$.
\end{corollary}
\begin{proof}
There exists normal embedding  (normal  injective monomorphism)  $\varphi :\,\,{{G}_{k}}\to {{B}_{k}}$  \cite{Heinek} such that $~~{{G}_{k}}\triangleleft {{B}_{k}}$. Indeed, according to Lemma \label{order_of_G_k}
 index $\left| {{B}_{k}}:~~{{G}_{k}} \right|=2$ so it is normal subgroup that is quotient subgroup  $~~{}^{{{B}_{k}}}/{}_{{{C}_{2}}}\simeq {{G}_{k}}$.
\end{proof}

\begin{theorem}
For any $k\geq 1$ we have $G_k \simeq Syl_2 A_{2^k}$.
\end{theorem}
\begin{proof} Group $C_2$ acts on the set $X = \{1,2\}$. Therefore we can recursively define sets $X^k$ on which group $B_k$ acts $X^1 = X,$  $X^k = X^{k-1} \times X \mbox{ for k>1}$.
At first we define $S_{2^k} = Sym(X^{k})$ and $A_{2^k} = Alt(X^{k})$ for all integer $k\geq 1$. Then  $ G_k <B_k < S_{2^k}$ and $A_{2^k} < S_{2^k}$.

We already know \cite{SkIrred} that $B_k \simeq Syl_2 (S_{2^k})$.
 Since $|A_{2^k}| = |S_{2^k}|/2$ therefore $|Syl_2 A_{2^k}| = |Syl_2 S_{2^k}|/2 = |B_k|/2$. By Lemma~\ref{order_of_G_k} it follows that $|Syl_2 A_{2^k}| = |G_k|$. Therefore it is left to show that $G_k < Alt(X^k)$.

Let us fix some $g = (g_1, g_2)\sigma^i$ where $g_1, g_2 \in B_{k-1}$, $i\in \{0, 1\}$ and $g_1 g_2 \in G_{k-1}$. Then we can represent $g$ as follows
\[
g = (g_1 g_2, e) \cdot (g_2\m, g_2) \cdot (e, e,)\sigma^i.
\]
In order to prove this theorem it is enough to show that $(g_1 g_2, e), (g_2\m, g_2),  (e, e,)\sigma \in Alt(X^k)$.

Element $(e, e,)\sigma$ just switch letters $x_1$ and $x_2$ for all $x \in X^{k}$. Therefore $(e, e,)\sigma$ is product of $|X^{k-1}| = 2^{k-1}$ transpositions and therefore $(e, e,)\sigma \in Alt(X^k)$.

Elements  $g_2\m$ and $g_2$ have the same cycle type. Therefore elements $(g_2\m, e)$ and $(e, g_2)$ also have the same cycle type. Let us fix the following cycle decompositions
\begin{align*}
(g_2\m, e) = \sigma_1 \cdot \ldots \cdot \sigma_{n},\\
(e, g_2) = \pi_1 \cdot \ldots \cdot \pi_{n}.
\end{align*}
Note that element $(g_2\m, e)$ acts only on letters like $x_1$ and element $(e, g_2)$ acts only on letters like $x_2$. Therefore we have the following cycle decomposition
\begin{align*}
(g_2\m, g_2) = \sigma_1 \cdot \ldots \cdot \sigma_{n} \cdot \pi_1 \cdot \ldots \cdot \pi_{n}.
\end{align*}
So, element $(g_2\m, g_2)$ has even number of odd permutations and then $(g_2\m, g_2) \in Alt(X^k)$.

Note that $g_1 g_2 \in G_{k-1}$ and $G_{k-1} = Alt(X^{k-1})$ by induction hypothesis.  Therefore $g_1g_2 \in Alt(X^{k-1})$. As elements $g_1 g_2$ and $(g_1 g_2, e)$ have the same cycle type then $(g_1 g_2, e) \in Alt(X^k)$.
\end{proof}

As it was proven by the author in \cite{SkIrred} Sylow 2-subgroup has structure $B_{k-1} \ltimes W_{k-1} $, where definition of $B_{k-1}$ is the same that was given in \cite{SkIrred}.

Recall that it was denoted
 by $W_{k-1}$ the subgroup of $Aut X^{[k]}$ such that has active states only on $ X^{k-1}$ and number of such states is even, i.e. $W_{k-1} \vartriangleleft  St_{G_k}(k-1)$ \cite{Ne}.
It was proven that the size of ${{W}_{k-1}}$ is equal to ${{2}^{{{2}^{k-1}}-1}},\,\,k > 1$ and its structure is $(C_2)^{{{2}^{k-1}}-1}$. The following structural theorem characterizing the group $G_k$ was proved by us \cite{SkIrred}.

\begin{theorem}
A maximal 2-subgroup of $Aut{{X}^{\left[ k \right]}}$ that acts by even permutations on ${{X}^{k}}$ has the structure of the semidirect product $G_k \simeq  B_{k-1} \ltimes W_{k-1} $ and isomorphic to $Syl_2A_{2^k}$.
\end{theorem}

Note that ${{W}_{k-1}}$ is subgroup of stabilizer of  ${{X}^{k-1}}$ i.e. ${{W}_{k-1}}<St_{Aut{X}^{[k]}}(k-1)\lhd AutX^{[k]}$ and is normal too $W_{k-1}\lhd AutX^{[k]}$, because conjugation keeps a cyclic structure of permutation so even permutation maps in even. Therefore such conjugation induce an automorphism of ${W}_{k-1}$ and $G_k \simeq B_{k-1}\ltimes W_{k-1}$.


\begin{remark}
As a consequence, the structure founded by us in
\cite{SkIrred} fully consistent with the recursive group representation based on the concept of wreath recursion \cite{Lav}.
\end{remark}

\begin{theorem}\label{_comm_F_k_eq_[L_k,F_k]}
Elements of $\aut[k]'$ have the following form $\aut[k]'=\{[f,l]\mid f\in B_k, l\in G_k\}=\{[l,f]\mid f\in B_k, l\in G_k\}$.
\end{theorem}
\begin{proof}
It is enough to show either $B_k'=\{[f,l]\mid f\in B_k, l\in G_k\}$ or $B_k'=\{[l,f]\mid f\in B_k, l\in G_k\}$ because if $f = [g,h]$ then $f\m = [h,g]$.

We prove the proposition by induction on $k$. For the case $k=1$ we have $B_1' = \langle e \rangle$.

Consider case $k>1$.
According to Lemma~\ref{form of comm_2} and Corollary \ref{c_2_wr_b_elem_repr} every element $w\in\aut[k]'$ can be represented as
\begin{align*}
w=(r_1, r_1\m [f,g])
\end{align*}
for some $r_1,f\in \aut[k-1]$ and $ g\in \syl[k-1]$ (by induction hypothesis). By the Corollary~\ref{c_2_wr_b_elem_repr} we can represent $w$ as commutator of
\begin{align*}
(e,a_{1,2})\sigma \in \aut[k] \mbox{ and } (a_{2,1}, a_{2,2}) \in \aut[k],
\end{align*}
where
\begin{align*}
a_{2,1} &= (f\m)^{r_1\m},\\
a_{2,2} &= r_{1} a_{2,1},\\
a_{1,2} &= g^{a_{2,2}\m}.
\end{align*}
 If $g \in G_{k-1}$ then by the definition of $G_k$ and Corollary~\ref{G_k_is_normal_in_B_k} we obtain $(e,a_{1,2})\sigma \in \syl[k]$.
\end{proof}

\begin{remark}\label{_comm_B_k_eq_[b_k,b_k]}
Let us to note that Theorem~\ref{_comm_F_k_eq_[L_k,F_k]} improve Corollary~\ref{cw_syl_p_s_p_k_eq_1_and_syl_p_a_p_k_eq_1} for the case $Syl_2 S_{2^k}$.
\end{remark}


\begin{proposition}\label{B'_k and B^2_k}
If $g$ is an element of the group $B_k $ then $g^2 \in B'_{k}$.
\end{proposition}
\begin{proof}
Induction on $k$. We  note that $B_k = B_{k-1} \wr C_2$. Therefore we fix some element
\[
g= (g_1, g_2)\sigma^i \in B_{k-1} \wr C_2,
\]
where $g_1, g_2 \in B_{k-1}$ and $i\in \{0, 1\}$. Let us to consider $g^2$ then two cases are possible:
\begin{align*}
g^2 = (g_1^2, g_2^2) \mbox{ or } g^2 = (g_1 g_2, g_2 g_1)
\end{align*}
In second case we consider a product of coordinates
$g_1 g_2 \cdot g_2 g_1 = g_1^2 g_2^2 x$. Since according to the induction hypothesis $g_i^2 \in B_k'$, $i\leq 2$ then $g_1 g_2 \cdot g_2 g_1 \in B_k'$ also according to Lemma~\ref{form of comm} $x\in B_k'$.
Therefore a following inclusion holds $(g_1 g_2, g_2 g_1) = g^2 \in B_k'$.
In first case the proof is even simpler because $g_1^2, g_2^2 \in B'$ by the induction hypothesis.
\end{proof}

\begin{lemma}\label{L_k_comm_criteria}
If an element $g=(g_1, g_2) \in G_k'$ then $g_1,g_2 \in G_{k-1}$ and $g_1g_2\in B_{k-1}'$.
\end{lemma}


\begin{proof}
As $B_k' < G_k$ therefore it is enough to show that  $g_1 \in G_{k-1}$ and $g_1g_2\in B_{k-1}'$.  Let us fix some $g=(g_1, g_2) \in G_k' < B_k'$. Then Lemma~\ref{form of comm} implies that $g_1 g_2 \in B_{k-1}'$.

In order to show that $g_1\in G_{k-1}$ we firstly consider just one commutator of arbitrary elements from $G_k$
\begin{align*}
f =(f_1, f_2)\sigma, \,\, h = (h_1, h_2)\pi \in G_k,
\end{align*}
where $f_1, f_2, h_1, h_2 \in B_{k-1}$, $\sigma, \pi \in C_2$. The definition of $G_k$ implies that $f_1 f_2, h_1 h_2 \in G_{k-1}$.

If $g= (g_1, g_2) = [f, h]$ then
\[
g_1  = f_1 h_{i} f_{j}\m h_{k}\m
\]
for some $i, j, k \in \{1,2\}$.
 Then
\begin{align*}
g_1 = f_1 h_{i} f_{j} (f_{j}\m)^2 h_{k} (h_{k}\m)^2  = (f_1 f_{j}) (h_{i} h_{k}) x (f_{j}\m h_{k}\m)^{2},
\end{align*}
where $x$ is product of commutators of $f_i, \, h_j$ and $f_i, \, h_k$, hence $x \in B_{k-1}'$.

It is enough to consider first product $f_1 f_{j}$.
 If $j=1$ then $f_1^2\in B'_{k- 1}$ by Proposition~\ref{B'_k and B^2_k} if $j=2$ then $f_1f_2 \in G_{k-1}$ according to definition of $G_k$, the same is true for $h_i h_k$. Thus, for any $i, j, k$ it holds $f_1 f_j, h_i h_k \in G_{k-1}$.  Besides that a square $(f_{j}\m h_{k}\m)^{2} \in B'_k$ according to Proposition~\ref{B'_k and B^2_k}.
Therefore $g_1 \in G_{k-1}$ because of Proposition~\ref{B'_k and B^2_k} and Proposition~\ref{comm_F_k_is_subgroup_of_L_k}, the same is true for $g_2 $.

Now it lefts to consider the product of some $f =(f_1, f_2), h = (h_1, h_2)$, where $f_1, h_1 \in G_{k-1}$, $f_1 h_1 \in G_{k-1}$ and $f_1 f_2, h_1 h_2 \in B_{k-1}'$
\begin{align*}
fh = (f_1 h_1, f_2 h_2).
\end{align*}

Since $f_1 f_2, h_1 h_2 \in B_{k-1}'$ by imposed condition in this item and taking into account that $f_1 h_1 f_2 h_2 = f_1 f_2 h_1 h_2 x$ for some $x\in B_{k-1}'$ then $f_1 h_1 f_2 h_2 \in B_{k-1}'$ by Lemma \ref{form of comm}. Other words closedness by multiplication holds and so according Lemma\ref{form of comm} we have element of commutator $G'_k$.
\end{proof}

In the following theorem we prove $2$ facts at once.
\begin{theorem}\label{_comm_G_k_eq_[G_k,G_k]}
The following statements are true.
\begin{itemize}
\item[1.] An element $g = (g_1, g_2) \in G_k'$ iff $g_1,g_2 \in G_{k-1}$ and $g_1g_2\in B_{k-1}'$.
\item[2.] Commutator subgroup $G'_k$ coincides with set of all commutators for $k\geq 1$
\[
\syl[k]'=\{[f_1,f_2]\mid f_1\in \syl[k], f_2\in \syl[k]\}.
\]
\end{itemize}

\end{theorem}

\begin{proof}
For the case $k=1$ we have $G_1' = \langle e \rangle$. So, further we consider the case $k\geq 2$.

Sufficiency of the first statement of this theorem follows from the Lemma~\ref{L_k_comm_criteria}. So, in order to prove necessity of the both statements it is enough to show that element
\begin{align*}
w=(r_1, r_1\m x),
\end{align*}
where $r_1  \in G_{k-1}$ and $ x \in B'_{k-1}$, can be represented as a commutator of elements from $G_k$.
%
%
By Proposition~\ref{_comm_F_k_eq_[L_k,F_k]} 
we have $ x = [f,g]$ for some $f\in\aut[k-1]$ and $g\in\syl[k-1]$. Therefore
\begin{align*}
w=(r_1, r_1\m [f,g]).
\end{align*}

By the Corollary~\ref{c_2_wr_b_elem_repr} we can represent $w$ as a commutator of
\begin{align*}
(e,a_{1,2})\sigma \in \aut[k] \mbox{ and } (a_{2,1}, a_{2,2}) \in \aut[k],
\end{align*}
where $a_{2,1} = (f\m)^{r_1\m}, a_{2,2} = r_{1} a_{2,1}, a_{1,2} = g^{a_{2,2}\m}.$
It only lefts to show that  $(e,a_{1,2})\sigma, \\ (a_{2,1}, a_{2,2}) \in G_k$.
Note the following 
\begin{align*}
a_{1,2} = g^{a_{2,2}\m} &\in G_{k-1}\mbox{ by Corollary ~\ref{G_k_is_normal_in_B_k}}.\\
a_{2,1} a_{2,2} = a_{2,1} r_1 a_{2,1} = r_1 [r_1, a_{2,1}] a_{2,1}^2 &\in \syl[k-1]\mbox{ by Proposition~\ref{comm_F_k_is_subgroup_of_L_k} and Proposition~\ref{B'_k and B^2_k}}.
\end{align*} 
So we have $(e,a_{1,2})\sigma \in \syl[k]$ and $(a_{2,1}, a_{2,2}) \in \syl[k]$ by the definition of $G_k$.
\end{proof} 

\begin{proposition}\label{g_sq_in_G_k}
For arbitrary $g\in G_k$ the inclusion $g^2\in G_k'$ holds.
\end{proposition}

\begin{proof}
Induction on $k$:  elements of $G^2_1$ have form $ (\sigma)^2 = e $, where $\sigma = (1,2)$, so the statement holds. In general case, when $k>1$, the elements of $G_k$ have the form $g = (g_1, g_2) \sigma^i,  \, g_1, g_2 \in B_{k-1}, \, i \in \{0,1\}$. Then we have two possibilities:  $g^2 = (g_1^2, g_2^2)\mbox{ or } g^2 = (g_1g_2, g_2g_1).$

Firstly we show that $g_1^2 \in G_{k-1},
g_2^2 \in G_{k-1}.$
 According to Proposition \ref{B'_k and B^2_k}, we have $g_1^2, g_2^2 \in B_{k-1}'$ and according to Proposition~\ref{comm_F_k_is_subgroup_of_L_k}, we have $B_{k-1}' < G_{k-1}$ then using Theorem \ref{_comm_G_k_eq_[G_k,G_k]}  $g^2=(g_1^2, g_2^2)  \in G_{k}$.

Consider the second case $g^2 = (g_1g_2, g_2g_1)$.
  Since $g\in G_k$, then, according to the definition of $G_k$ we have that $g_1 g_2  \in G_{k-1}$.  By Proposition~\ref{comm_F_k_is_subgroup_of_L_k}, and definition of $G_k$, we obtain
\begin{gather*}
 g_2g_1 = g_1g_2 g_2\m g_1\m g_2 g_1 = g_1g_2 [g_2\m, g_1\m]\in G_{k-1}, \\ 
g_1g_2 \cdot g_2 g_1 = g_1 g_2^2 g_1 = g_1^2 g_2^2 [g_2^{-2}, g_1\m] \in B_{k-1}'.
\end{gather*}
Note that $g_1^2, g_2^2 \in B'_{k-1}$ according to Proposition \ref{B'_k and B^2_k}, then $g_1^2 g_2^2 [g_2^{-2}, g_1\m] \in B_{k-1}'$. 
Since $g_1g_2 \cdot g_2 g_1 \in B'_{k-1}$ and $g_1g_2,  g_2 g_1 \in G_{k-1}$, then, according to Lemma \ref{L_k_comm_criteria}, we obtain $g^2 = (g_1g_2, g_2g_1)\in G'_k$.
\end{proof}

\begin{statment}\label{FrPr}
The commutator subgroup is a subgroup of $G_{k}^{2}$  i.e. $G{{'}_{k}}<G_{k}^{2}$.
\end{statment}

\begin{proof}
Indeed, an arbitrary commutator presented as product of squares. Let  $a,\,\,b\in G$ and set that
$x=a,\,\,\,y={{a}^{-1}}ba,\,\,\,z={{a}^{-1}}{{b}^{-1}}$. Then ${{x}^{2}}{{y}^{2}}{{z}^{2}}={{a}^{2}}{{({{a}^{-1}}ba)}^{2}}{{({{a}^{-1}}{{b}^{-1}})}^{2}}=ab{{a}^{-1}}{{b}^{-1}},\,\,$
in more detail:
${{a}^{2}}{{({{a}^{-1}}ba)}^{2}}{{({{a}^{-1}}{{b}^{-1}})}^{2}}={{a}^{2}}{{a}^{-1}}ba\,{{a}^{-1}}ba\,\,{{a}^{-1}}{{b}^{-1}}{{a}^{-1}}{{b}^{-1}}= \\ = abb{{b}^{-1}}{{a}^{-1}}{{b}^{-1}}=\left[ a,b \right]$.
In such way we obtain all commutators and their products.
Thus, we generate by squares the whole $G{{'}_{k}}$.
\end{proof}

\begin{corollary}\label{Fr}
For the Syllow subgroup $(Syl_2 A_{2^k})$ the following equalities $Syl'_2 A_{2^k}=(Syl_2 A_{2^k})^2$, $\Phi(Syl_2 A_{2^k}) = Syl'_2 A_{2^k}$, that are characteristic properties of special p-groups \cite{DWard}, are true.

\end{corollary}
\begin{proof}
As well known, for an arbitrary group (also by  Statement \ref{FrPr})  the following  embedding  $G'\triangleleft {{G}^{2}}$  holds.
In view of the above Proposition \ref{g_sq_in_G_k}, a reverse embedding for $G_k$ is true.
Thus, the group $Syl_2 A_{2^k}$ has some properties of special $p$-groups that is $P'= \Phi(P)$ \cite{DWard} because $G_k^2=G'_k$ and so Frattini subgroup $\Phi(Syl_2 A_{2^k})=Syl'_2 (A_{2^k})$.
\end{proof}


\begin{corollary}
Commutator width of the group $Syl_2 A_{2^k}$ equals to $1$ for $k\geq 2$.
\end{corollary}

It immediately follows from item 2 of Theorem \ref{_comm_G_k_eq_[G_k,G_k]}.

\end{section}

\section{Minimal generating set}
For the construction of minimal generating set we used the representation of elements of group $G_k$ by portraits of automorphisms at restricted binary tree $Aut X^k$.
 For convenience we will identify elements of $G_k$ with its faithful representation by portraits of automorphisms from $AutX^{[k]}$.

We denote by $A|_l$ a set of all functions
$a_l$, such, that
$[\varepsilon,\ldots,\varepsilon,a_l,\varepsilon,\ldots]\in
[A]_l$.
Recall that, according to \cite{SysB}, $l$-coordinate subgroup $U<G$ is the following subgroup.

\begin{definition}  For an arbitrarry $k\in \mathbb{N}$ we call a $k-$\emph{coordinate} subgroup $U<G$  a subgroup, which is determined by $k$-coordinate sets $[U]_l$, $l\in \mathbb{N}$, if this subgroup consists of all Kaloujnine's tableaux $a\in I$
for which $[a]_l \in [U]_l$.
\end{definition}

We denote by ${{G}_{k}}(l)$ a level subgroup of $G_k$, which consists of the tuples of v.p. from ${{X}^{l}}$, $l<k-1$ of any $\alpha \in G_k$.
We denote as ${{G}_{k}}(k-1)$ such subgroup of $G_k$ that is generated by v.p., which are located on ${{X}^{k-1}}$ and isomorphic to ${{W}_{k-1}}$. Note that ${{G}_{k}}(l)$ is in bijective correspondence (and isomorphism) with $l$-coordinate subgroup $[U]_l$ \cite{SysB}.

For any v.p. ${{g}_{li}}$ in ${{v}_{li}}$ of ${{X}^{l}}$ we set in correspondence with ${{g}_{li}}$ the permutation $\varphi \left( {{g}_{li}} \right)\in {{S}_{2}}$ by the following  rule:

\begin{equation}\label{hom}
\varphi ({{g}_{li}}) = \left\{
\begin{array}{rl}
 (1,2), & \ \  \mbox{if} \ \ {{g}_{li}}\ne e,\\
     e, & \ \  \mbox{if} \ \ {{g}_{li}}=e.
\end{array} \right.
\end{equation}
					

Define a homomorphic map from ${{G}_{k}}(l)$ onto ${{S}_{2}}$ with the kernel consisting of all products of even number of transpositions that belongs to ${{G}_{k}}(l)$. For instance, the element $(12)(34)$ of $G_k (2)$ belongs to $ ker\varphi $.
Hence, $\varphi \left( {{g}_{li}} \right)\in {{S}_{2}}$.

\begin{definition}
We define the subgroup of $l$-th level as a subgroup generated by all possible vertex permutation of this level.
\end{definition}

\begin{statment}\label{comm_str1}
In ${{G}_{k}}'$, the following $k$ equalities are true:
\begin{gather} \label{parity}
\prod\limits_{l=1}^{{{2}^{l}}}{\varphi ({{g}_{lj}})}=e ,\,\,\,\,0\le \,l<k-1.
\end{gather}

For the case $i=k-1$,
the following condition holds:

\begin{gather}\label{ident}
\prod\limits_{j=1}^{{{2}^{k-2}}}{\varphi ({{g}_{k-1j}})}=\prod\limits_{j={{2}^{k-2}}+1}^{{{2}^{k-1}}} \varphi ({{g}_{k-1j}})=e.   		
 \end{gather}

Thus, $G{{'}_{k}}$ has $k$ new conditions on a combination of level subgroup elements, except for the condition of last level parity from the original group.
\end{statment}

\begin{proof}
Note that the condition (\ref{parity}) is compatible with that were founded by R. Guralnik in \cite{Gural_2010}, because as it was proved by author \cite{SkIrred} ${{G}_{k-1}}\simeq {{B}_{k-2}}\rtimes {{\mathcal{W}}_{k-1}}$, where ${{B}_{k-2}}\simeq \underset{i=1}{\overset{k-2}{\mathop{\wr }}}\,C_{2}^{(i)}$.

 According to Property \ref{FrPr}, $G{{'}_{k}} \leq G_{k}^{2}$, so it is enough to prove the statement for the elements of $G_{k}^{2}$. Such elements, as it was described above, can be presented in the form  $s= ({{s}_{l1}},...,{{s}_{l{{2}^{l}}}})\sigma $,  where $\sigma \in G_{l-1}^{{}}$ and ${{s}_{li}}$ are states of $s\in {{G}_{k}}$ in ${{v}_{li}}$, $i\le {{2}^{l}}$.
For convenience we will make the transition from the tuple $({{s}_{l1}},...,{{s}_{l{{2}^{l}}}})$ to the tuple $({{g}_{l1}},...,{{g}_{l{{2}^{l}}}})$.  Note that there is the trivial vertex permutation $g_{lj}^{2}=e$ in the product of the states ${{s}_{lj}}\cdot {{s}_{lj}}$.

Since in $G{{'}_{k}}$ v.p. on ${{X}^{0}}$ are trivial,  so $\sigma $ can be decomposed as $\sigma =\left( {{\sigma}_{11}},\,{{\sigma }_{21}} \right)$, where ${{\sigma }_{21}},\,{{\sigma }_{22}}$ are root permutations in ${{v}_{11}}$ and ${{v}_{12}}$.

Consider the square of $s$. So we calculate squares ${{\left( \left( {{s}_{l1}},{{s}_{l2}},...,{{s}_{l{{2}^{l-1}}}} \right)\sigma  \right)}^{2}}$. The condition (\ref {parity}) is equivalent to the condition that ${{s}^{2}}$ has even index on each level. Two cases are feasible: if permutation $\sigma =e$, then ${{\left( \left( {{s}_{l1}},{{s}_{l2}},...,{{s}_{l{{2}^{l-1}}}} \right)\sigma  \right)}^{2}}=\left( s_{l1}^{2},s_{l2}^{2},...,s_{l{{2}^{l-1}}}^{2} \right)e$, so after the transition  from $\left( s_{l1}^{2},s_{l2}^{2},...,s_{l{{2}^{l-1}}}^{2} \right)$ to $\left( g_{l1}^{2},g_{l2}^{2},...,g_{l{{2}^{l-1}}}^{2} \right)$, we get a tuple of trivial permutations $\left( e,\,\,...\,\,,e \right)$ on ${{X}^{l}}$, because $g_{lj}^{2}=e$.  In general case, if $\sigma \ne e$,  after such transition we obtain $\left( {{g}_{l1}}{{g}_{l\sigma (2)}},\,\,...\,\,,\,{{g}_{l{{2}^{l-1}}}}{{g}_{l\sigma ({{2}^{l-1}})}} \right)\sigma _{{}}^{2}$. Consider the product of form
\begin{gather}\label{multres}
\prod\limits_{j=1}^{{{2}^{l}}}{\varphi ({{g}_{lj}}{{g}_{l\sigma (j)}})},
\end{gather}
					
where $\sigma $ and ${{g}_{li}}{{g}_{l\sigma (i)}}$  are from $\left( {{g}_{l1}}{{g}_{l\sigma (2)}},\,\,...\,\,,\,{{g}_{l{{2}^{l-1}}}}{{g}_{l\sigma ({{2}^{l-1}})}} \right)\sigma _{{}}^{2}$.

Note that
 each element ${{g}_{lj}}$ occurs twice in (\ref{multres}) regardless of the permutation $\sigma $, therefore considering commutativity of homomorphic images $\varphi ({{g}_{lj}}),\,\,\,1\le j\le {{2}^{l}}$ we conclude that
$\prod\limits_{j=1}^{{{2}^{l}}}{\varphi ({{g}_{lj}}{{g}_{l\sigma (j)}})}=\prod\limits_{j=1}^{{{2}^{l}}}{\varphi (g_{lj}^{2})=e},$ because of $g_{lj}^{2}=e$. We rewrite $\prod\limits_{j=1}^{{{2}^{l}}}{\varphi (g_{lj}^{2})=e}$ as characteristic condition:
$\prod\limits_{j=1}^{{{2}^{l-1}}}{\varphi ({{g}_{lj}})}=\prod\limits_{j={2}^{l-1}+1}^{{{2}^{l}}}{\varphi (g_{lj})=e}.$


According to Property \ref{FrPr}, any commutator from $G{{'}_{k}}$ can be presented as a product of some squares ${{s}^{2}},\,\,s\in {{G}_{k}}$, $s=( ({{s}_{l1}},...,{{s}_{l{{2}^{l}}}})\sigma \,)$.

A product of elements of $G_{k}(k-1)$ satisfies the equation $\prod\limits_{j=1}^{{{2}^{l}}}{\varphi ({{g}_{lj}})}=e$, because any permutation of elements from $X^k$, which belongs to $G_k$ is even.
Consider the element $s= ({{s}_{k-1,1}},...,{{s}_{k-1,{{2}^{k-1}}}})\sigma  $, where $ ({{s}_{k-1,1}},...,{{s}_{k-1,{{2}^{k-1}}}}) \in G_k(k-1)$, $\sigma \in G_{k-1}$.
If $g_{01}=(1,2)$, where $g_{01}$ is root permutation of $\sigma$, then $s^2= ({{s}_{k-1,1}{s}_{k-1\sigma (1)}},..., s_{{k-1}, ({{2}^{k-1}})}{s_{k-1, \sigma ({{2}^{k-1}})}})$, where $\sigma(j)>2^{k-1}$ for $j \leq 2^{k-1}$, and if $j < 2^{k-1}$ then $\sigma(j) \geq 2^{k-1}$. Because of $\prod\limits_{j=1}^{2^{k-1}} \varphi (g_{k-1,j})=e$ in $G_k$ and the property $\sigma(j) \leq 2^{k-1}$ for $j > 2^{k-1}$,
then the product $\prod\limits_{j=1}^{{{2}^{k-2}}}{\varphi ({{g}_{k-1,j}}{g}_{k-1,\sigma (j)}})$ of images of v.p. from $({g_{k-1,1} g_{k-1,\sigma (1)}},...,{{g_{k-1,  ({{2}^{k-1}})}} g_{k-1, \sigma ({{2}^{k-1}})}})$ is equal to $\prod\limits_{j=1}^{{{2}^{k-1}}}{\varphi ({{g}_{k-1,j}}})=e$. Indeed in $\prod\limits_{j=1}^{2^{k-1}} \varphi (g_{k-1,j})$ and as in $\prod\limits_{j=1}^{{{2}^{k-1}}}{\varphi ({{g}_{k-1,j}}{g}_{k-1,\sigma (j)}})$ are the same v.p. from $X^{k-1}$ regardless of such $\sigma$ as described above.

  The same is true for right half of $X^{k-1}$. Therefore the equality (\ref{ident}) holds.


  Note that such product $\prod\limits_{j=1}^{2^{k-1}} \varphi (g_{k-1,j})$ is homomorphic image of $({{g}_{l,1}{g}_{l,\sigma (1)}},...,{{{g}_{l,  ({{2}^{l}})}}{g}_{l \sigma ({{2}^{l}})}})$, where $l=k-1$, as an element of $G'_k(l)$ after mapping (\ref{hom}).

 If $g_{01} =e$, where $g_{01}$ is root permutation of $\sigma$ then $\sigma $ can be decomposed as $\sigma =\left( {{\sigma}_{11}},\,{{\sigma }_{12}} \right)$, where ${{\sigma }_{11}},\,{{\sigma }_{12}}$ are root permutations in ${{v}_{11}}$ and ${{v}_{12}}$. As a result $s^2$ has a form  $\left( ({{s}_{l1}{s}_{l\sigma (1)}},...,{{s}_{l \sigma ({{2}^{l-1}})}}) \sigma_1^2,  ({{s}_{l 2^{l-1}+ 1}{s}_{l \sigma (2^{l-1}+ 1)}},...,{{s}_{l  ({{2}^{l}})}}{{s}_{l \sigma ({{2}^{l}})}})  \sigma_2^2 \right)$, where $l=k-1$. As a result of action of ${{\sigma }_{11}}$ all states of $l$-th level with number $1 \leq j \leq 2^{k-2}$ permutes in coordinate from $1$ to $2^{k-2}$ the other are fixed. The action of ${{\sigma }_{11}}$ is analogous.

  It corresponds to the next form of element from $G'_k(l)$: $ ({{g}_{l1}{g}_{l\sigma_1 (1)}},...,{{g}_{l \sigma_1 ({{2}^{l-1}})}}), \\ ({{g}_{l 2^{l-1}+ 1}{g}_{l \sigma_2 (2^{l-1}+ 1)}},...,{{g}_{l  ({{2}^{l}})}}{{g}_{l \sigma_2 ({{2}^{l}})}})  $.
Therefore the product of form
$\prod\limits_{j=1}^{{{2}^{k-2}}}{\varphi ({{g}_{k-1,j}}{{g}_{l\sigma (j)}})}=\prod\limits_{j=2^{k-2}+1}^{{{2}^{k-1}}}{\varphi (g_{k-1,j}^{2})=e},$ because of $g_{k-1,j}^{2}=e$.
Thus, characteristic equation (\ref{ident}) of $k-1$ level holds.


 The conditions (\ref{parity}), (\ref{ident}) for  every ${{s}^{2}}, s\in G_k$ hold so it holds for their product that is  equivalent to conditions  hold for every commutator.
\end{proof}

\begin{definition}
We define a subdirect product of group ${{G}_{k-1}}$ with itself by equipping it with condition (\ref{parity}) and  (\ref{ident}) of index parity on all of $k-1$ levels.
\end{definition}

\begin{corollary}\label{Subdir}
  The subdirect product ${{G}_{k-1}}\boxtimes {{G}_{k-1}}$ is defined by  $k-2$ outer relations on level subgroups. The order of ${{G}_{k-1}}\boxtimes {{G}_{k-1}}$ is ${{2}^{{{2}^{k}}-k-2}}$.
\end{corollary}
\begin{proof}
 We specify a subdirect product for the group ${{G}_{k-1}}\boxtimes {{G}_{k-1}}$  by using $(k-2)$ conditions for the subgroup levels. Each ${{G}_{k-1}}$ has even index on $k-2$-th level, 
it implies that its relation  for $l=k-1$ holds automatically. This occurs because of the conditions of parity for the index of the last level is characteristic of each of the multipliers ${{G}_{k-1}}$. Therefore It is not an essential condition for determining a subdirect product.

  Thus, to specify  a subdirect product in the group ${{G}_{k-1}}\boxtimes {{G}_{k-1}}$, there are obvious only $k-2$ outer conditions on subgroups of levels.   Any of such conditions reduces the order of ${{G}_{k-1}}\times {{G}_{k-1}}$  in 2 times. Hence, taking into account that the order of ${{G}_{k-1}}$  is ${{2}^{{{2}^{k-1}}-2}}$, the order of ${{G}_{k-1}}\boxtimes {{G}_{k-1}}$ as a subgroup of ${{G}_{k-1}}\times {{G}_{k-1}}$ the following: $\left| {{G}_{k-1}}\boxtimes {{G}_{k-1}} \right|={{\left( {{2}^{{{2}^{k-1}}-2}} \right)}^{2}}:{{2}^{k-2}}={{2}^{{{2}^{k}}-4}}:{{2}^{k-2}}={{2}^{{{2}^{k}}-k-2}}$.
	Thus, we use $k-2$ additional conditions on level subgroup to define the subdirect product ${{G}_{k-1}}\boxtimes {{G}_{k-1}}$, which contain $G{{'}_{k}}$ as a proper subgroup of ${{G}_{k}}$. Because according to the conditions, which are realized in the commutator of $G{{'}_{k}}$, (\ref{ident}) and (\ref{parity}) indexes of levels are even.
\end{proof}

\begin{corollary} \label{emb} A commutator $G{{'}_{k}}$ is embedded as a normal subgroup in ${{G}_{k-1}}\boxtimes {{G}_{k-1}}$.
\end{corollary}
\begin{proof}
A proof  of injective embedding $G{{'}_{k}}$ into ${{G}_{k-1}}\boxtimes {{G}_{k-1}}$ immediately  follows from last item of proof of  Corollary \ref{Subdir}.   The minimality of $G{{'}_{k}}$ as a normal subgroup of ${{G}_{k}}$ and injective embedding $G{{'}_{k}}$ into ${{G}_{k-1}}\boxtimes {{G}_{k-1}}$ immediately  entails that $G{{'}_{k}}\triangleleft {{G}_{k-1}}\boxtimes {{G}_{k-1}}$.
\end{proof}

\begin{theorem} \label{isomCom}
A commutator of ${{G}_{k}}$ has form $G{{'}_{k}} = {{G}_{k-1}}\boxtimes {{G}_{k-1}}$, where the subdirect product is defined by relations (\ref{parity}) and (\ref{ident}). The order of $G{{'}_{k}}$  is ${{2}^{{{2}^{k}}-k-2}}$.
\end{theorem}

\begin{proof}
 Since according to Statement \ref{comm_str1} $({{g}_{1}},{{g}_{2}})$ as elements of $G{{'}_{k}}$ also satisfy relations (\ref{parity}) and (\ref{ident}), which define the subdirect product ${{G}_{k-1}}\boxtimes {{G}_{k-1}}$.
Also condition ${{g}_{1}}{{g}_{2}}\in B{{'}_{k-1}}$ gives parity of permutation which defined by $({{g}_{1}},{{g}_{2}})$ because $B{{'}_{k-1}}$ contains only element with even index of level \cite{SkIrred}.
The group $G{{'}_{k}}$ has 2 disjoint domains of transitivity so $G{{'}_{k}}$ has the structure of a subdirect product of ${{G}_{k-1}}$ which acts on this domains transitively.
Thus, all elements of $G{{'}_{k}}$ satisfy the conditions (\ref{parity}), (\ref{ident}) which define subdirect product ${{G}_{k-1}}\boxtimes {{G}_{k-1}}$. Hence $G{{\text{ }\!\!'\!\!\text{ }}_{k}}<{{G}_{k-1}}\boxtimes {{G}_{k-1}}$ but $G{{\text{ }\!\!'\!\!\text{ }}_{k}}$ can be equipped by some other relations,  therefore, the presence of isomorphism has not yet been proved. For proving revers inclusion we have to show that every element from ${{G}_{k-1}}\boxtimes {{G}_{k-1}}$ can be expressed as word ${{a}^{-1}}{{b}^{-1}}ab$, where $a,b\in {{G}_{k}}$.
Therefore, it suffices to show the reverse inclusion. For this goal we use that $G{{\text{ }\!\!'\!\!\text{ }}_{k}}<{{G}_{k-1}}\boxtimes {{G}_{k-1}}$.  As it was shown in \cite{SkIrred} that  the order of ${{G}_{k}}$ is ${{2}^{{{2}^{k}}-2}}$.

As it was shown above, $G{{'}_{k}}$ has $k$ new conditions relatively to ${{G}_{k}}$. Each condition is stated on some level-subgroup. Each of these conditions reduces an order of the corresponding level subgroup in 2 times, so the order of $G{{'}_{k}}$ is in ${{2}^{k}}$ times lesser. On every ${{X}^{l}}$,  $l\le k-1$, there is even number of active v.p. by this reason, there is trivial permutation on ${{X}^{0}}$.

According to the Corollary \ref{Subdir}, in the subdirect product ${{G}_{k-1}}\boxtimes {{G}_{k-1}}$ there are exactly $k-2$ conditions relatively to ${{G}_{k-1}}\times {{G}_{k-1}}$, which are for the subgroups of levels. It has been shown that the relations (\ref{parity}), (\ref{ident}) are fulfilled in $G{{'}_{k}}$.

Let ${{\alpha }_{lm}}$, $0\le l\le k-1$, $0\le m \le 2^{l-1}$ be an automorphism from ${{G}_{k}}$ having only one active v.p. in ${{v}_{lm}}$, and let ${{\alpha }_{lm}}$ have trivial permutations in rest of the vertices. Recall that partial case of notation of form ${{\alpha }_{lm}}$ is the generator ${\alpha }_{l}:={{\alpha }_{l1}}$ of ${{G}_{k}}$ which was defined by us in \cite{SkIrred} and denoted by us as ${{\alpha }_{l}}$. Note that the order of ${{\alpha }_{li}},\,\,\,0\le l\le k-1$ is 2. Thus, ${{\alpha }_{ji}}=\,\alpha _{ji}^{-1}$. We choose a generating set consisting of the following $2k-3$ elements: $({{\alpha }_{1,1;2}}),{{\alpha }_{2,1}},...,{{\alpha }_{k-1,1}},{{\alpha }_{2,3}},...,{{\alpha }_{k-{{1,2}^{k-2}}+1}}$, where $({{\alpha }_{1,1;2}})$ is an automorphism having exactly 2 active v.p. in ${{v}_{11}}$ and ${{v}_{12}}$. Product of the form $({{\alpha}_{j1}}{{\alpha}_{l1}}{{\alpha}_{j1}}){{\alpha}_{l1}}$ are denoted by ${{P}_{lm}}$. In more details, ${{P}_{lm}}={{\alpha }_{ji}}{{\alpha }_{lm}}{{\alpha }_{ji}}{{\alpha }_{lm}}$,  where ${{\alpha }_{ji}}\in {{G}_{k}}(j)$. Using a conjugation by generator ${{\alpha}_{j}}$, $0\le j<l$ we can express any v.p. on $l$-level, because $({{\alpha }_{j}}{{\alpha }_{l}}{{\alpha }_{j}})=\,{{\alpha}_{l{{2}^{l-j-1}}+1}}$. Consider the product ${{P}_{lj}}=({{\alpha }_{j}}{{\alpha }_{l}}{{\alpha }_{j}}){{\alpha }_{l}}$.

	\begin{enumerate}
   \item We need to show that every element of ${{G}_{k-1}}\boxtimes {{G}_{k-1}}$
     can be constructed as ${{g}^{-1}}{{h}^{-1}}gh$,  $g,\,h\in {{G}_{k}}$.  This proves the absence of other relations in $G{{'}_{k}}$ except those that in the subdirect product ${{G}_{k-1}}\boxtimes {{G}_{k-1}}$. Thereby we prove the embeddedness of $G{{'}_{k}}$ in ${{G}_{k-1}}\boxtimes {{G}_{k-1}}$. We have to construct an element of form ${{P}_{k-1}}{{P}_{k-2}}\cdot ... \cdot {{P}_{1}}{{P}_{0}}$ as a product of elements of form $[g,h]$, where ${{P}_{l}}=\prod\limits_{i=1}^{{{2}^{l}}}{{{P}_{lm}}}$ satisfying  relations (\ref{parity}), (\ref{ident}).

    \item
	We have to construct an arbitrary tuple of 2 active v.p. on ${{X}^{l}}$ as a product of several ${{P}_{l}}$. We use the generator ${{\alpha}_{l}}$ and conjugating  it by ${{\alpha}_{j}}$,  $j<l$.  It corresponds to the tuple of v.p. of the form $({{g}_{l1}},e,...,e,{{g}_{lj}},e,...,e)$,  where ${{g}_{l1}},\,\,{{g}_{lj}}$ are non-trivial. Note that this tuple $({{g}_{l1}},e,...,e,{{g}_{lj}},e,...,e)$ is an element of direct product if we consider as an element of $S_2$ in vertices of $X^l$.
To obtain a tuple of v.p. of form   $(e,...,e,{{g}_{lm}},e,...,e,{{g}_{lj}},e,...,e)$ we multiply ${{P}_{lj}}$ and ${{P}_{lm}}$.

   \item To obtain a tuple of v.p. with $2m$ active v.p. we construct $~\prod\limits_{i=1}^{m}{{{P}_{l{{j}_{i}}}}},\,m<{{2}^{l}}$ for varying $i,j<{{2}^{k-2}}$.
  \end{enumerate}

On the $(k-1)$-th level we choose the generator $\tau $ which was defined in \cite{SkIrred} as $\tau =\tau _{k-1,\,1} {\tau _{k-1,\,\,{{2}^{k-1}}}}$. Recall that it was shown in \cite{SkIrred} how to express  any ${{\tau }_{ij}}$ using $\tau $, ${{\tau }_{{i,{2}^{k-2}}}}$, ${{\tau }_{{j, {2}^{k-2}}}}$, where $\,i,j<{{2}^{k-2}}$, as a product of commutators ${{\tau }_{ij}}={{\tau }_{i, {2}^{k-2}}}{{\tau }_{j, {2}^{k-2}}}=(\alpha _{i}^{-1}\tau _{{{1,2}^{k-2}}}^{-1}\alpha _{i}^{{}}\tau _{{{j,2}^{k-2}}}^{{}})$.
  Here ${{\tau }_{i{{,2}^{k-2}}}}$ was expressed as the commutator ${{\tau }_{i{{,2}^{k-2}}}}=\alpha _{i}^{-1}\tau _{{{1,2}^{k-2}}}^{-1}\alpha _{i}^{{}}\tau _{{{1,2}^{k-2}}}^{{}}$. Thus, we express all tuples of elements satisfying to relations
 (\ref{parity}) and (\ref{ident})
 by using only commutators of ${{G}_{k}}$. Thus, we get all tuples of each level subgroup elements satisfying the relations (\ref{parity}) and (\ref{ident}). It means we express every element of each level subgroup by a commutators.
In particular to obtain a tuple of v.p. with $2m$ active v.p. on ${{X}^{k-2}}$ of ${{v}_{11}}{{X}^{[k-1]}}$, we will construct the product for ${{\tau }_{ij}}$ for varying $i,j<{{2}^{k-2}}$.

Thus, all vertex labelings of automorphisms, which appear in the representation of ${{G}_{k-1}}\boxtimes {{G}_{k-1}}$ by portraits as the subgroup of $Aut{{X}^{[k]}}$, are also in the representation of $G{{'}_{k}}$.

Since there are
faithful representations of ${{G}_{k-1}}\boxtimes {{G}_{k-1}}$ and ${G'}_{k}$ by
portraits of automorphisms from $Aut{{X}^{[k]}}$, which coincide with each other, then subgroup ${G'}_{k}$ of ${{G}_{k-1}}\boxtimes {{G}_{k-1}}\simeq G{{'}_{k}}$ is equal to whole ${{G}_{k-1}}\boxtimes {{G}_{k-1}}$  ( i.e. ${{G}_{k-1}}\boxtimes {{G}_{k-1}} = G{{'}_{k}}$).
\end{proof}






The archived results are confirmed by algebraic system GAP calculations. For instance, $\left| Sy{{l}_{2}}{{A}_{8}}\right| ={{2}^{6}}={{2}^{{{2}^{3}}-2}}$ and $\left| (Syl{{A}_{{{2}^{3}}}})' \right| = {{2}^{{{2}^{3}}-3-2}} = 8$.
The order of $G_2$ is 4, the number of additional relations in subdirect product is $k-2=3-2=1$.
Then we have the same result $(4 \cdot 4):2^1=8$, which confirms Theorem \ref{isomCom}.

\begin{example}
  Set $k=4$ then $|(Syl A_{16})'| = |(G_4)'| = 1024$, $|G_3| = 64$, since $ k-2=2$, so according to our theorem above order of $Syl_2 A_{16} \boxtimes Syl_2 A_{16}$ is defined by $2^{k-2}=2^2$ relations, and by this reason is equal to $(64\cdot 64) :4=1024$. Thus, orders are coincides.
\end{example}
\begin{example}
The true order of $(Syl_2 {A_{32}})'$ is $33554432 = 2^{25}$, $k=5$. A number of additional relations which define the subdirect product is $k-2=3$.
Thus, according to Theorem \ref{isomCom},
$ \mid (Syl_2 A_{16} \boxtimes Syl_2 A_{16})'\mid =2^{14} 2^{14}: 2^{5-2} = 2^{28} : 2^{5-2}= 2^{25}$.
\end{example}


According to calculations in GAP we have:  $Sy{{l}_{2}}{{A}_{7}}\simeq Sy{{l}_{2}}{{A}_{6}}\simeq {{D}_{4}}$.  Therefore its derived subgroup  $\left( Sy{{l}_{2}}{{A}_{7}} \right)'\simeq \left( Sy{{l}_{2}}{{A}_{6}} \right)'\simeq \left( {{D}_{4}} \right)'={{C}_{2}}$.

The following structural law for Syllows 2-subgroups is typical. The structure of $Sy{{l}_{2}}{{A}_{n}},\,\,Sy{{l}_{2}}{{A}_{k}}$ is the same. If for all $n$ and  $k$ that have the same multiple of 2 as multiplier in decomposition on $n!$ and $k!$  Thus, $Sy{{l}_{2}}{{A}_{2k}}\simeq Sy{{l}_{2}}{{A}_{2k+1}}$.

\begin{example}  $Sy{{l}_{2}}{{A}_{7}}\simeq Sy{{l}_{2}}{{A}_{6}}\simeq {{D}_{4}}$, $Sy{{l}_{2}}{{A}_{10}}\simeq Sy{{l}_{2}}{{A}_{11}}\simeq Sy{{l}_{2}}{{S}_{8}}\simeq \left( {{D}_{4}}\times {{D}_{4}} \right)\rtimes {{C}_{2}}$.
$Syl_2 A_{12} \simeq  Syl_2 S_{8}\boxtimes Syl_2 S_{4}$, by the same reasons that from the proof of Corollary \ref{Subdir} its commutator subgroup is decomposed as $(Syl_2 A_{12})' \simeq  (Syl_2 S_{8})' \times (Syl_2 S_{4})'$.
\end{example}

\begin{lemma}
\label{comm_str}
In $G''_{k}$ the following equalities are true:

\begin{gather} \label{parity2}
	\prod\limits_{j=1}^{{{2}^{l-2}}}{\varphi ({{g}_{lj}})}=\prod\limits_{j={{2}^{l-2}}+1}^{{{2}^{l-1}}}\varphi ({{g}_{lj}})\,\,=\,\prod\limits_{j={{2}^{l-1}}+1}^{{{2}^{l-1}}+{{2}^{l-2}}}{\varphi ({{g}_{lj}})}=\prod\limits_{j={{2}^{l-1}}+{{2}^{l-2}}+1}^{{{2}^{l}}}{\varphi ({{g}_{lj}}}),\,\,\,\,\,2<l<k
\end{gather}

In case $l=k-1$, the following conditions hold:

\begin{gather}\label{prod_K-1}
\prod\limits_{j=1}^{{{2}^{l-2}}}{\varphi ({{g}_{lj}})}=\prod\limits_{j={{2}^{i-1}}+1}^{{{2}^{l-1}}}{\varphi ({{g}_{lj}})=e,\,\,\,\,\,\prod\limits_{j={{2}^{l-1}}}^{{{2}^{l-1}}+{{2}^{l-2}}}{\varphi ({{g}_{lj}})}=\prod\limits_{j={{2}^{l-1}}+{{2}^{l-2}}}^{{{2}^{l}}}{\varphi ({{g}_{lj}})}}=e
\end{gather}

In other terms, the subgroup  $G''_{k}$ has an even index of any level of ${{v}_{11}}{{X}^{\left[ k-2 \right]}}$ and of ${{v}_{12}}{{X}^{\left[ k-2 \right]}}$.
\end{lemma}

\begin{proof}

As a result of derivation of $G'_k$, elements of $G''_k(1)$ are trivial.
Due the fact that $G{{'}_{k}}\simeq {{G}_{k-1}}\boxtimes {{G}_{k-1}}$, we can derivate $G{{'}_{k}}$ by commponents. The commutator of ${{G}_{k-1}}$  is already investigated in Theorem \ref{isomCom}. As $G_{k-1}^{2}=G{{'}_{k-1}}$ by Corollary \ref{Fr}, it is more convenient to present a characteristic equalities in the second commutator $G'{{'}_{k}}\simeq G{{'}_{k-1}}\boxtimes G{{'}_{k-1}}$ as equations in $G_{k-1}^{2}\boxtimes G_{k-1}^{2}$. As shown above, for $2\leq l<k-1$, in $G_{k-1}^{2}$ the following equalities are true:
\begin{gather}
\prod\limits_{j=1}^{{{2}^{l-1}}}{\varphi ({{g}_{lj}}{{g}_{l\sigma (j)}})}= \prod\limits_{j=1}^{{{2}^{l-1}}}{\varphi ({{g}_{lj}})}     \prod\limits_{j=1}^{{{2}^{l-1}}}{\varphi ({{g}_{l\sigma (j)}})} = \prod\limits_{j=1}^{{{2}^{l-1}}}{\varphi ({{g}_{lj}})}     \prod\limits_{j=1}^{{{2}^{l-1}}}{\varphi ({{g}_{li}})} =\prod\limits_{j=1}^{{{2}^{l-1}}}{\varphi (g_{lj}^{2})=e}
 \end{gather}
  \begin{gather}  \label{level_cond}
\prod\limits_{j=1}^{{{2}^{l-2}}}{\varphi ({{g}_{lj}})}=\prod\limits_{j={{2}^{l-2}}+1}^{{{2}^{l-1}}}{\varphi ({{g}_{lj}}})\,\,=\,\prod\limits_{j={{2}^{l-1}}+1}^{{{2}^{l-1}}+{{2}^{l-2}}}{\varphi ({{g}_{lj}})}=\prod\limits_{j={{2}^{l-1}}+{{2}^{l-2}}+1}^{{{2}^{l}}}{\varphi ({{g}_{lj}}}).
\end{gather}
 The equality (\ref{level_cond}) is true because of it is the initial group $G{{'}_{k}}\simeq {{G}_{k-1}}\boxtimes {{G}_{k-1}}$. The equalities
  \[\prod\limits_{j={{2}^{l-1}}+1}^{{{2}^{l-1}}+{{2}^{l-2}}}{\varphi ({{g}_{lj}})}=\prod\limits_{j={{2}^{l-1}}+{{2}^{l-2}}+1}^{{{2}^{l}}}{\varphi ({{g}_{lj}}})\]

   are right for elements of second group $G{{'}_{k-1}}$, since elements of the original group are endowed with this conditions.

Upon a squaring of $G{{'}_{k}}$ any element of $G{{'}_{k}}(l)$, satisfies the equality \eqref{level_cond} in addition to satisfying the previous conditions (\ref{parity2}) because of ${{\left( G{_{k-1}}(l) \right)}^{2}} = G{{'}_{k-1}}(l)$. The similar conditions appears in ${{\left( G{{'}_{k-1}}(k-2) \right)}^{2}}$ after squaring of $G{{'}_{k}}$. Thus, taking into account the characteristic equations of $G{{'}_{k-1}}(l)$, the subgroup ${{\left( G{{'}_{k-1}}(k-2) \right)}^{2}}$ satisfies the equality:

\begin{gather} \label{prod_K-1}
\prod\limits_{j=1}^{{{2}^{k-3}}}{\varphi ({{g}_{lj}})}=\prod\limits_{j={{2}^{k-3}}+1}^{{{2}^{k-2}}}{\varphi ({{g}_{lj}})=e,\,\,\,\,\,\prod\limits_{j={{2}^{k-2}}+1}^{{{2}^{k-2}}+{{2}^{k-3}}}{\varphi({{g}_{lj}})}=\prod\limits_{j={{2}^{k-1}}+{{2}^{k-2}}+1}^{{{2}^{k-1}}}{\varphi ({{g}_{lj}})}}=e.
\end{gather}

	Taking into account the structure $G{{'}_{k}}\simeq {{G}_{k-1}}\boxtimes {{G}_{k-1}}$ we obtain after derivation $G'{{'}_{k}}\simeq ({{G}_{k-2}}\boxtimes {{G}_{k-2}})\boxtimes ({{G}_{k-2}}\boxtimes {{G}_{k-2}})$.
With respect to conditions \ref{parity}, \ref{ident}  in the subdirect product we have that the order of  $G'{{'}_{k}}$ is ${{2}^{{{2}^{k}}-k-2}}:{{2}^{2k-3}}={{2}^{{{2}^{k}}-3k+1}}$ because on every level $2\le l<k$ order of level subgroup $G'{{'}_{k}}(l)$  is in 4 times lesser then order of  $~G{{'}_{k}}(l)$. On the 1-st level one new condition arises that reduce order of  $~G{{'}_{k}}(1)$ in 2 times. Totally we have  $2(k-2)+1=2k-3$ new conditions in comparing with $~G{{'}_{k}}$.
\end{proof}

\begin{example} Size of $(G_4'')$ is 32, a size of direct product $(G_3')^2$ is 64, but, due to relation on second level of $G''_k$, the direct product $(G_3')^2 $ transforms into the subdirect product $G_3'\boxtimes G_3'$ that has 2 times less feasible combination on $X^2$. The number of additional relations in the subdirect product is $k-3=4-3=1$. Thus the order of product is reduced in $2^1$ times.
\end{example}


\begin{example}
The  commutator  subgroup of $Syl'_2(A_8)$  consists of  elements:
$\{ e, (13)(24)(57)(68), \\ (12)(34), (14)(23)(57)(68), (56)(78), (13)(24)(58)(67),(12)(34)(56)(78), (14)(23)(58)(67)\}$.
The commutator $Syl_2 '(A_8) \simeq C_2 ^3$ that is an elementary abelian 2-group of order 8. This fact confirms our formula $d(G_k)=2k-3$, because $k=3$ and $d(G_k)= 2k-3=3$.
A minimal  generating  set  of  $Syl'_2(A_{8})$ consists of 3 generators: $(1, 3)(2, 4)(5, 7)(6, 8), (1, 2)(3, 4), \\ (1, 3)(2, 4)(5, 8)(6, 7).$ 
\end{example}

\begin{example}
The minimal generating  set  of  $Syl'_2(A_{16})$ consists of 5 (that is $2\cdot4-3$) generators:   $(1, 4, 2, 3)(5, 6)(9, 12)(10, 11),
(1, 4)(2, 3)(5, 8)(6, 7), (1, 2)(5, 6), \\ (1, 7, 3, 5)
(2,8,4,6)(9,14,12,16)(10,13,11,15),  (1,7)
(2, 8)(3, 6)(4, 5)(9, 16, 10, 15)\times \\ \times(11, 14, 12, 13)$.
\end{example}

\begin{example}
Minimal generating  set  of  $Syl'_2(A_{32})$ consists of 7 (that is $2\cdot5-3$) generators:  $(23,24)(31,32), (1,7)(2,8)(3,5,4,6)(11,12)(25,32)(26,31)(27,29)(28,30),  \\   (3,4)(5,8)(6,7)(13,14)(23,24)(27,28)(29,32)(30,31),
\, \, (7,8)(15,16)(23,24)(31,32), \\  (1,9,7,15)(2,10,8,16)(3,11,5,13)(4,12,6,14)(17,
    29,22,27,18,30,21,28) \times \\ (19,32,23,26,20,31,24,25),
 (1,5,2,6)(3,7,4,8)(9,15)(10,16)(11,13)(12,14)(19,20)\times \\ (21,24,22,23)(29,31)(30,
    32), (3,4)(5,8)(6,7)(9,11,10,12)(13,14)(15,16)\times \\ (17,23,20,22,18,24,19,
    21) (25,29,27,32,26,30,28,31).
$
\end{example}
This confirms our formula of minimal generating set size  $2\cdot k-3$.

\begin{corollary}
A total number of irreducible generic sets of $(Sy{{l}_{2}}{{A}_{{{2}^{k}}}})'$ is $\left( {{2}^{2k-3}}-1 \right)\left( {{2}^{2k-3}}-{{2}^{1}} \right)\cdot \,\,...\,\,\,\cdot \left( {{2}^{2k-3}}-{{2}^{2k-4}} \right):(2k-3)!$
\end{corollary}

It follows from the fact that Frattini quotient of the commutator subgroup is an elementary abelian 2-group in this case. It can be considered as vector space which base has $2k-3$ generating vectors. Taking into consideration that permutation of generating vectors do not give us a new base we have to reduce the number of generating vectors in $(2k-3)!$ times.

Let elements $g,\,\,h\in {{G}_{k}}$ are conjugated that is ${{x}^{-1}}gx=\,\,h$ where $x\in {{G}_{k}}$.

\begin{remark}
The order of commutator subgroup according to Corollary \ref{Subdir} is ${{2}^{{{2}^{k}}-k-2}}$  that is in ${{2}^{k}}$ times lesser then order of Syllow 2-subgroup that is ${{2}^{{{2}^{k}}-2}}$. Since if we find that subgroup elements $g,\,\,h$ belongs to one commutator subgroup then it reduces the complexity of solving conjugacy search problem in ${{2}^{k}}$ times.
\end{remark}

The minimal generating set for ${{G}_{4}}$ can be presented in form of wreath recursion:    \[{{a}_{1}}=(e,e)\sigma ,\text{ }{{b}_{2}}=\left( {{a}_{1}},e \right),\text{ }{{b}_{3}}=\left( {{b}_{2}},e \right),\text{ }{{b}_{4}}=\left( {{b}_{3}},\,\,{{b}_{3}} \right),\]
where $ \sigma =(1,2)$.
The minimal generating set for $G{{'}_{4}}$ can be presented in form of wreath recursion:  \[{{a}_{2}}=(\sigma ,\,\sigma ),\text{ }{{a}_{3}}=(e,{{a}_{2}}),{{a}_{4}}=\left( {{a}_{3}},{{a}_{3}} \right),~\,\,{{b}_{3}}=(e,{{b}_{2}}),{{b}_{4}}=({{b}_{3}},{{b}_{3}}).\]
Where $\sigma,\,{{a}_{3}},{{a}_{4}}$ generators of the first multiplier ${{G}_{3}}$ and $\sigma, {{b}_{3}},{{b}_{4}}$ generators of the second. 


\section{Conclusion }

The size of minimal generating set for commutator of Sylow 2-subgroup of alternating group $A_{2^k}$ was proven is equal to $2k-3$.

A new approach to presentation of Sylow 2-subgroups of alternating group ${A_{{2^{k}}}}$ was applied. As a result the short proof of a fact that commutator width of Sylow 2-subgroups of alternating group ${A_{{2^{k}}}}$, permutation group ${S_{{2^{k}}}}$ and Sylow $p$-subgroups of $Syl_2 A_{p^k}$ ($Syl_2 S_{p^k}$) are equal to 1 was obtained.
 Commutator width of permutational wreath product $B \wr C_n$ were investigated.


\end{document}